\documentclass[12pt,a4paper,psamsfonts]{amsart}
\usepackage{amssymb,amscd,amsxtra,calc}
\usepackage{cmmib57}
\usepackage[dvips,a4paper,bookmarks,bookmarksopen,%
bookmarksnumbered,pdfauthor={Viehweg and Zhang},%
colorlinks=false]{hyperref}

\theoremstyle{plain}
    \newtheorem{theorem}{Theorem}[section] %
    \renewcommand{\thetheorem}%
    {\arabic{section}.\arabic{theorem}}
    \newtheorem{corollary}[theorem]{Corollary}
    \newtheorem{proposition}[theorem]{Proposition}
    \newtheorem{question}[theorem]{Question}
    \newtheorem{lemma}[theorem]{Lemma}
    
\theoremstyle{definition}
    \newtheorem{setup}[theorem]{}
    \newtheorem{set-up}[theorem]{Set-up}
    \newtheorem{remark}[theorem]{Remark}
    
    \newtheorem{assumption}[theorem]{Assumption}
    \newtheorem{rem-ass}[theorem]{Remark and Assumption}
    \newtheorem{claim}[theorem]{Claim}   
    
    \newtheorem{construction}[theorem]{Construction}
\numberwithin{equation}{section}

\renewcommand{\P}{\mathord{\mathbb P}}
\newcommand{\Q}{\mathord{\mathbb Q}}

\newcommand{\Z}{\mathord{\mathbb Z}}

\renewcommand{\AA}{\mathord{\mathcal A}}

\newcommand{\OO}{\mathord{\mathcal O}}

\newcommand{\A}{{\text{\rm A}}}
\newcommand{\DCC}{{\text{\rm DCC}}}

\newcommand{\Ell}{{\text{\rm Ell}}}

\newcommand{\klt}{\text{\rm klt}}
\newcommand{\lcm}{\text{\rm lcm}}

\newcommand{\Supp}{\text{\rm Supp}}

\begin{document}
\title[Effective Iitaka fibrations]{Effective Iitaka fibrations}
\author{Eckart Viehweg}
\address{
\textsc{Universit\"at Duisburg-Essen \endgraf
Fachbereich Mathematik \endgraf
45117 Essen, Germany}}
\email{viehweg@uni-due.de}
\author{De-Qi Zhang}
\address{
\textsc{Department of Mathematics} \endgraf
\textsc{National University of Singapore \endgraf
2 Science Drive 2, Singapore 117543, Singapore
}}
\email{matzdq@nus.edu.sg}
\begin{abstract}
For every $n$-dimensional projective manifold $X$ of Kodaira dimension $2$
we show that $\Phi_{|M K_X|}$ is birational to an Iitaka fibration
for a {\it computable} positive integer $M = M(b, B_{n-2})$,
where $b > 0$ is minimal with $|bK_F| \ne \emptyset$ for a general fibre $F$ 
of an Iitaka fibration of $X$, and where $B_{n-2}$ is the Betti number of a 
smooth model of the canonical $\Z/(b)$-cover of
the $(n-2)$-fold $F$. In particular, $M$ is a universal constant if the dimension $n \le 4$.
\end{abstract}


\thanks{This work has been supported by the DFG-Leibniz program and by the SFB/TR 45
``Periods, moduli spaces and arithmetic of algebraic varieties''.
The second author is partially supported by an academic research fund of NUS}
\maketitle
\raggedbottom
Building up on the work of H. Tsuji, C.D.~Hacon and J. McKernan in \cite{HM} and 
independently S. Takayama in \cite{Ta} have shown the existence of a constant $r_n$
such that $\Phi_{|mK_X|}$ is a birational map for every $m \ge r_n$
and for every complex projective $n$-fold $X$ of general type. 
\par
If the Kodaira dimension $\kappa=\kappa(X)$ is non-negative and $\kappa < n$, consider an Iitaka fibration $f : X \to Y$, i.e. a rational map onto a projective manifold $Y$ of dimension $\kappa$ with a connected general fibre $F$ of Kodaira dimension zero. We define the {\it index} $b$ of $F$ to be 
$$
b= \min \{b' > 0 \,\, | \,\, |b' K_{F}| \ne \emptyset\},
$$ 
and $B_{n-\kappa}$ to be the $(n-\kappa)$-th Betti number of
a nonsingular model of the $\Z/(b)$-cover of $F$, obtained by taking the $b$-th root 
out of the unique member in $|bK_{F}|$, or as we will say, the 
{\it middle Betti number} of the {\it canonical covering} of $F$. 
\begin{question}\label{quest}
Is there a constant $M := M(n, \kappa, b, B_{n-\kappa})$ such that
$\Phi_{|MK_X|}$ is (birational to) an Iitaka fibration
$f:X\to Y$ for all projective $n$-folds $X$ of Kodaira dimension $\kappa$?
\end{question}
Assume that for all $s \le n$ there exists an effective constant $a(s)$ 
such that for every projective $s$-fold $V$ of non-negative Kodaira dimension, one has $|a(s)K_V|\ne \emptyset$ and such that the dimension of $|a(s)K_V|$ is at least one if $\kappa(V)>0$. 
Then J. Koll\'ar gives in \cite[Th 4.6]{Ko86} a formula for the constant $M$ 
in \ref{quest} in terms of $a(s)$ and $n$.

Question \ref{quest} has been answered in the affirmative
by Fujino-Mori \cite{FM} for $\kappa = 1$. In this note we show that the answer
is also affirmative for $\kappa = 2$. 
\begin{theorem} \label{ThA}
Let $X$ be an $n$-dimensional projective manifold of Kodaira dimension $2$
with Iitaka fibration $f : X \to Y$. Then there exists 
a computable positive integer $M$ depending only on the index $b$ of a general fibre
$F$ of $f$ and on the middle Betti number $B_{n-2}$ of the canonical covering of $F$,
such that $\Phi_{|M K_X|}$ is birational to $f$.
\end{theorem}
We say a few words about the proof.
In Section \ref{aux} we will consider two $\Q$-divisors $D_Y$ and $L_Y$ on $Y$ 
(with $Y$ suitably chosen),
such that the reflexive hull of $f_*\OO_X(b N K_{X/Y})$ is isomorphic to $\OO_Y(b N (D_Y + L_Y))$
for some constant $N$ depending on $B_{n-\kappa}$. The divisor $b N (K_Y+D_Y + L_Y)$ is big, and
in order to prove Theorem \ref{ThA} it remains to bound a multiple $M$ of $bN$ for which $\Phi_{|M K_Y +MD_Y + M L_Y|}$ is birational. 

The two divisors $L_Y$ and $D_Y$ are of different nature. $D_Y$ is given by the multiplicities of
fibres of $f$, and the pair $(X,D_Y)$ is klt, whereas $bL_Y$ is the semistable part of the direct image of $bK_{X/Y}$, hence nef. We will apply the log minimal model program for surfaces in various ways, to reduce the problem to the case where on some birational model $W$ the corresponding divisor $K_W + D_W + L_W$ is nef. The rest then will be easy if the sheaf $K_Y + D_Y$ is big, and will not be too difficult if $L_Y$ is big and
$K_Y + D_Y$ pseudoeffective. For the remaining cases we will consider in Section \ref{prf} 
the {\it pseudo-effective threshold}, i.e. the smallest real number $e$ with $K_W + D_W + e L_W$ pseudo-effective, and we will show that $e$ is bounded away from one. 

Starting from Section \ref{log} our arguments and methods are using $\dim(Y)=\kappa(X)=2$,
sometimes just by laziness, and at other places in an essential way. 
After finishing the proof of Theorem
\ref{ThA} we will be a bit more precise (see Remarks \ref{end} and \ref{rem2.1}).

If one assumes that the general fibre $F$ of the Iitaka fibration has a good minimal model $F'$, hence 
one with $b K_{F'}=0$, then $L_Y$ would be the pullback of a nef and big $\Q$-divisor on some compactification of a moduli scheme. As we will discuss in 
Remark \ref{rem2.1}, assuming the existence of good minimal models, the existence of nice compactifications of moduli schemes might lead to an affirmative answer to Question \ref{quest}. 

Conjecturally the index $b$ and the Betti number $B_{\dim(F)}$ should be bounded by a constant 
depending only on the dimension of $F$. So one could hope for an affirmative answer to:
\begin{question}\label{quest2}
Can one choose the constant $M$ in Question \ref{quest} to be independent of $b$ and $B_{n-\kappa}$?
\end{question}
For example, for $F$ an elliptic curve one has $b=1$ and $B_{1}=2$.
For surfaces $F$ of Kodaira dimension zero, the index $b$ divides $12$, and $22$ is an upper bound for the middle Betti number $B_{2}$ of the smooth minimal model of the canonical covering of $F$. 
Hence for $n\leq 4$ the constant $M$ in Theorem \ref{ThA} can be chosen to be {\it universal}, i.e. 
only depending on $n$. Since by \cite[\S 10]{Mo}, \cite[Corollary 6.2]{FM}, \cite[Th 1.1]{CC}, \cite{HM} and \cite{Ta} the same holds true for $\dim(X)=3$ and $\kappa(X)=0, 1$, or $3$ we can state:
\begin{corollary}\label{CoB}
There is a computable universal constant $M_3$ such that $\Phi_{|M_3 K_X|}$
is an Iitaka fibration for every $3$-dimensional projective manifold $X$.
\end{corollary}
We remark that when $\dim X = 3$ and $\kappa(X) = 2$, Koll\'ar \cite[(7.7)]{Ko94} 
has already shown that there exists a universal constant $M'$ such that $H^0(X, mK_X) \ne 0$ for all $m \ge M'$, under the additional assumption that the Iitaka fibration is non-isotrivial.
A direct proof of Corollary \ref{CoB}, using the existence of good minimal models, will be
given at the end of Section \ref{com}.
\begin{footnote}
{After a first version of this article was submitted to the arXiv-server, we learned that the Corollary \ref{CoB} has been obtained independently by Adam T. Ringler in \cite{Ri}, using different arguments.}
\end{footnote}

\begin{setup}
{\bf Conventions.}
\end{setup}
We adopt the conventions of Hartshorne's book, of \cite{KMM} and and of \cite{KM}.
However, if $D$ is a $\Q$-divisor on $X$ we will often write $H^0(X,D)$ or
$H^0(X,\OO_X(D))$ instead of $H^0(X,\OO_X(\lfloor D \rfloor))$,
and write $|D|$ instead of $|\lfloor D \rfloor|$.
By abuse of notations we will not distinguish line bundles and 
linear equivalence classes of divisors.
\begin{setup}
{\bf Acknowledgments.}
\end{setup}

We are grateful to Osamu Fujino and to Noboru Nakayama for numerous discussions on the
background of Question \ref{quest}. In particular O. Fujino explained the results
in \cite{FM} in April 2007 in several e-mails to the second named author, and N.~Nakayama brought 
to our attention the references \cite{Ii} and \cite{KU} on effective results for surfaces.
We thank the referee for suggestions on how to improve the presentation of the main result and 
of the methods leading to its proof. 

The article was written during a visit of the second named author to the University Duisburg-Essen. He thanks the members of the Department of Mathematics in Essen for their support and hospitality.

\section{Some auxiliary results}\label{aux}

\begin{set-up}\label{setup2.1}
Let $X$ be a complex $n$-fold of Kodaira dimension $\kappa$. We will consider an Iitaka fibration 
$f:X\to Y$ of $X$ with $Y$ nonsingular, and $F$ will denote a general fibre of $f$. 
Replacing $X$ by some nonsingular blowup, as in \cite[\S 3]{Vi83} or \cite[\S 2, \S 4]{FM}, 
one may assume that $f: X \to Y$ is a morphism, that the discriminant of $f$ is contained 
in a normal crossing divisor of $Y$ and that each effective divisor $E$ in $X$, 
with ${\rm codim}_Y(f(E))\geq 2$, is exceptional for some morphism $X\to X'$ with $X'$ nonsingular. 
In particular, for all $i \ge 1$ and for all such divisors $E$
one has
$$
H^0(X, \, \OO_X(ibK_X)) = H^0(X, \, \OO_X(ibK_X+ E))=
H^0(Y, \, \OO_Y(ibK_Y)\otimes f_*\OO_X(ibK_{X/Y})^{\vee\vee}),$$
where $b$ denotes again the index of $F$ and where $f_*\OO_X(ibK_{X/Y})^{\vee\vee}$ is the
invertible sheaf, obtained as the reflexive hull of $f_*\OO_X(ibK_{X/Y})$.
By \cite[Corollary 2.5]{FM} one can define the semistable part of 
$f_*\OO_X(ibK_{X/Y})$ as a $\Q$-Cartier divisor $i L^{ss}_{X/Y}$, compatible with base change, such that
$\OO_Y(iL^{ss}_{X/Y})\subset f_*\OO_X(ibK_{X/Y})^{\vee\vee}$ for $i$ sufficiently divisible, 
and such that both sheaves coincide if $f:X\to Y$ is semistable in codimension one. 
In particular, $L^{ss}_{X/Y}$
is nef. We will write
$$
L_Y = \frac{1}{b} L^{ss}_{X/Y} \mbox{ \ \ and \ \ } D_Y = \sum_P \frac{s_P}{b} P
$$
for the $\Q$-divisors with $\OO_Y(ib(L_Y+D_Y))=f_*\OO_X(ibK_{X/Y})^{\vee\vee}$.
We remark that $D_Y$ is supported on the discriminant locus of $f$ and
$b(L_Y + D_Y)$ is only a $\Q$-divisor (whose denominators may not be uniformly bounded); 
see \cite[Proposition 2.2]{FM}.

Let $B_{n-\kappa}$ be the middle Betti number of the canonical covering of $F$, and
$$N = N(B_{n-\kappa}) = \lcm\{m \in \Z_{>0} \, | \, \varphi(m) \le B_{n-\kappa}\},$$
where $\varphi$ denotes the Euler $\varphi$-function. 
By \cite[Theorem 3.1]{FM}, $NbL_Y=NL^{ss}_{X/Y}$ is 
an integral Cartier divisor. By \cite[Proposition 2.8]{FM},
if $s_P\neq 0$, there exist $u_P, v_P \in \Z_{>0}$ with $0 < v_P \le bN$ such that  
$$0 < \frac{s_P}{b} = \frac{b N u_P - v_P}{bN u_P} < 1.$$

So all the non-zero coefficients of $D_Y$ are contained in
$$
\A(b, N) := \{\frac{bN u- v}{bN u} \, | \, u, v \in \Z_{>0}; \, 0 < v \le bN\}\setminus \{0\}.
$$
\end{set-up}

\begin{lemma}\label{form} In the Set-up \rm{\ref{setup2.1}}, the following hold true.
\begin{itemize}
\item[(1)] 
The set $\A(b, N)$ is a $\DCC$ set in the sense of \cite[\S 2]{AM},
and one has 
$$\frac{1}{N b} \le {\rm Inf} \, \A(b,N).$$
\item[(2)]
$(Y, D_Y)$ is $\klt$.
\item[(3)]
The $\Q$-divisor $K_Y + D_Y + L_Y$ is big.
\item[(4)]
For every $m \in \Z_{> 0}$, 
we have $H^0(X, mbK_X) \cong H^0(Y, mb(K_Y + D_Y + L_Y))$; further
the map $\Phi_{|mbK_X|}$ is birational to the Iitaka fibration $f$
if and only if $|mb(K_Y + D_Y + L_Y)|$ gives rise
to a birational map.
\item[(5)] $NbL_Y$ is an integral Cartier divisor.
\item[(6)]
If $s+1 \in \Z_{> 0}$ is divisible by $Nb$, then $(s+1)D_Y \ge \lceil sD_Y \rceil$.
\end{itemize}
\end{lemma}

\begin{proof}
Part (1) is obvious and (5) was mentioned already in Set-up \ref{setup2.1}.
For (2), we remark that $D_Y$, as part of the discriminant locus, is a simple normal crossing divisor
and that $s_P/b \in (0, 1)$. The parts (3) and (4) are obvious, since for all $i \ge 1$
$$H^0(X, \, ibK_X) = H^0(Y, \, ib(K_Y + L_Y + D_Y)).$$
Finally (6) is a consequence of the description of the coefficients of $D_Y$
as elements of $\A(b,N)$. In fact, since all $\beta=\frac{bNu-v}{bNu}\in\A(b,N)$ are larger than
or equal to $1-\frac{1}{u}$ and since $(s+1)\beta \in \frac{1}{u}\cdot\Z$, 
one finds that $(s+1)\beta \ge \lceil s \beta \rceil$.
\end{proof}
\section{Log minimal models of surfaces and pseudo-effectivity}\label{log}
From now on we will restrict ourselves to the case $\kappa=2$.

\begin{remark}\label{rem2.3} 
As we will see in proving Theorem \ref{ThA}, the constant $M(b, B_{n-2})$ 
(later written as $M(b, N)$) can be computed using the invariants
$\beta(\A)$ and $\epsilon(\A)$ of the $\DCC$ set $\A=\A(b,N)$ 
(see \cite[Th 4.12]{AM} and \cite[Complement 5.7.4]{Ko94}, or \cite[Th 5.4]{La}).
\end{remark}
\begin{lemma}\label{canmod}
There is a birational morphism $\sigma : Y \to W$ such that the following hold true.
\begin{itemize}
\item[(1)]
$K_W + D_W + L_W$ is ample and
$K_Y + D_Y + L_Y = \sigma^*(K_W + D_W + L_W) + E_{\sigma}.$
Here $D_W : = \sigma_*D_Y$, $L_W : = \sigma_*L_Y$;
$E_{\sigma} \ge 0$ is an effective $\sigma$-exceptional
divisor. 
\item[(2)]
$(W, D_W)$ is $\klt$.
\item[(3)]
Suppose that $L_S$ is big for $S = Y$ or for $S = W$. Then $L_S \sim_{\Q} L_S'$
with $(S, D_S + L_S')$ $\klt$.
\end{itemize}
\end{lemma}
Before giving the proof, let us recall the log minimal model program (LMMP) for surfaces.
\begin{construction}\label{LMMP} Starting with any two dimensional $\klt$-pair $(Y, \Delta_Y)$, 
the LMMP provides us with a sequence $\gamma : Y \to Z$ of contractions of ($K_Y + \Delta_Y$)-negative extremal rays. If $K_Y + \Delta_Y$ is big or more generally pseudo-effective
(which will be assumed here),
$\gamma$ will be birational. Writing $\Delta_Z=\gamma_*\Delta_Y$, the $\Q$-divisor
$E := K_Y+\Delta_Y-\gamma^* (K_Z+\Delta_Z)$ is effective and $\gamma$-exceptional,
the $\Q$-divisor $K_Z+\Delta_Z$ is nef and $(Z,\Delta_Z)$ is $\klt$.

By the abundance theorem for $\klt$ log surfaces (see \cite{Ko+}, for example), 
there is a morphism with connected fibres $\psi:Z \to W$ such that $K_Z+\Delta_Z$ 
is the pullback of an ample $\Q$-divisor $H$ on $W$. 
\end{construction}
\begin{proof}[Proof of Lemma \ref{canmod}]
As in the proof of \cite[Th 5.2]{FM}, the nefness of $L_Y$ and the bigness
of $K_Y + D_Y + L_Y$ allows us to find some $a \in \Q_{>0}$
and a $\klt$-pair $(Y, \Delta_Y)$ such that
$$K_Y + L_Y + D_Y \sim_{\Q} a(K_Y + \Delta_Y).$$

If $K_Y + \Delta_Y$ is big, as we assumed in Lemma \ref{canmod}, the morphism
$\sigma=\psi\circ\gamma:Y\to W$ in Construction \ref{LMMP} is birational and the ample $\Q$-divisor on $W$ is $H=K_W+\Delta_W$. So $K_Z+\Delta_Z=\psi^*(K_W+\Delta_W)$ and hence the $\gamma$-exceptional divisor $E$ 
is 
$$ K_Y + \Delta_Y - \sigma^*(K_W+\Delta_W).$$
In particular, $(W,\Delta_W)$ is again $\klt$. We find
\begin{multline*}
K_Y + L_Y + D_Y \sim_{\Q} a\sigma^*(K_W + \Delta_W) + a E =\\
\sigma^* \sigma_*(K_Y + L_Y + D_Y)  + a E =
\sigma^*(K_W + L_W + D_W) + a E.
\end{multline*}

\par
(1) is true by Lemma \ref{form} (3). Part (2) follows from (1) and Lemma \ref{form} (2).
\par
For part (3), we refere to \cite[Proposition 2.61, Corollary 3.5]{KM}.
\end{proof}
\begin{rem-ass}\label{wlg}
The morphism $\sigma : Y \to W$, constructed in Lemma \ref{canmod}, factors
through a minimal resolution $\widetilde{W}$ of $W$. So $\sigma=\pi \circ \eta$ for
$$
\begin{CD}
Y @>\xi>>  \widetilde{W} @>\pi >> W.
\end{CD}
$$  
Applying the construction in Section \ref{aux} to $\widetilde{W}$ instead of $Y$
one obtains divisors $L_{\widetilde W}$ and $D_{\widetilde W}$, which are just 
the direct images of the divisors $L_Y$ and $D_Y$. 

Using the calculation in Lemma \ref{canmod}, we see that Lemma \ref{form}
holds true even with $Y$ replaced by the minimal resolution of $W$.
By abuse of notations we will replace $Y$ by $\widetilde{W}$ and assume in the sequel that
the morphism $\sigma:Y \to W$ in Lemma \ref{canmod} is a minimal desingularization.

\end{rem-ass}

The answer to Question \ref{quest} is quite easy when $K_Y + D_Y$ is big, and
especially when $L_Y \equiv 0$.

\begin{lemma}\label{big}
Suppose that $K_Y + D_Y$ is big. Then there is a constant $M=M(b,N)$ such that for all
$s \ge M(b, N)$  with $s+1$ divisible by $Nb$, the linear system
$|(s+1) (K_Y + D_Y + L_Y)|$ defines a birational map. In particular, $\Phi_{|MK_X|}$
is an Iitaka fibration for some $M = M(b, N)$ depending only on the set
$\A(b, N)$, and hence only on $b$ and $N = N(B_{n-2})$.
\end{lemma}

\begin{proof}
Applying Construction \ref{LMMP} to $(Y, D_Y)$,
we get a birational morphism $\eta : Y \to Z$ such that $K_Z + D_Z = \eta_*(K_Y + D_Y)$ is ample
and $E_{\eta} := K_Y + D_Y -\eta^*(K_Z + D_Z)$ is an $\eta$-exceptional effective $\Q$-divisor.

Note that the coefficients of $D_Z = \eta_*D_Y$ still belong to
the same $\DCC$ set $\A(b, N)$. The Remark (3) on page  60 of \cite{La} allows to apply
\cite[Th 3.2]{La}. As in \cite[Th 5.3]{La} one finds a constant $M(b, N)$, depending only on
the set $\A(b, N)$, such that the linear system 
$$
|K_Y + \lceil s \eta^*(K_Z + D_Z)\rceil |
$$
gives rise to a birational map for every $s \ge M(b, N)$. 
The same \cite[Th 3.2]{La} applies to
$$
|K_Y + \lceil s(K_Y + D_Y) + (s+1) L_Y \rceil |,
$$ 
since $(s+1) L_Y$ is pseudo-effective
and hence the boundary divisor of the above adjoint linear system has nef part
larger than $s \eta^*(K_Z + D_Z)$.

Assume further that $Nb$ divides $s+1$. Then by Lemma \ref{form}
%
%
$$
K_Y + \lceil sK_Y + sD_Y + (s+1)L_Y \rceil \le (s+1) (K_Y + D_Y + L_Y).
$$ 
This implies the first part of Lemma \ref{big}. Now the second part follows from the first part 
using Lemma \ref{form} (4).
\end{proof}

\begin{lemma}\label{smalle}
Suppose that $L_Y$ is big and $K_Y + D_Y + eL_Y$ is pseudo-effective
for some $e \in [0, 1)$.
Then there is a constant $M = M(b, N, e)$ such that
$\Phi_{|MK_X|}$ is an Iitaka fibration.
\end{lemma}

\begin{proof}
Consider the Zariski decompositions
$$K_Y + D_Y + e L_Y = P_e + N_e \mbox{ \ \ and \ \ }
K_Y + D_Y + L_Y = P_Y + N_Y.$$
Then $P_Y \ge (1-e) L_Y + P_e$. Since $Nb L_Y$ is an integral divisor
$$
P_Y^2 \ge (1-e)^2L_Y^2 \ge \frac{(1-e)^2}{(Nb)^2}.
$$
For a very general curve $C_t$, we have
$$P_Y . C_t \ge (1-e) L_Y . C_t \ge \frac{1-e}{Nb}.
$$
Assume that $s(1-e) > 4Nb.$
Applying \cite[Th 3.2]{La} one finds that the adjoint linear system
$|K_Y + \lceil s(K_Y + D_Y + L_Y) + L_Y \rceil|$ 
(whose boundary divisor has the nef part larger than $sP_Y$)
gives rise to a birational map.
Assume further that $Nb$ divides $(s+1)$. The lemma follows 
from the observation that the latter system is included in the following
(see Lemma \ref{form}):
$$
|(s+1)(K_Y + D_Y + L_Y)|.
$$
\end{proof}

The most difficult part of the proof of Theorem \ref{ThA} is the one where
$K_Y + D_Y$ is not pseudo-effective, and hence where $L_Y$ is not numerically trivial.
As a first step, we will need the following construction, well known to experts as
a consequence of \cite{Ba} (see however \cite{Ar}). This will be essential in the next section;
for the completeness and for the need of the precise description of the end product (i.e., $V$),
we give a proof. 

\begin{proposition}\label{threshold}
Suppose that $K_Y + D_Y$ is not pseudo-effective.
Then in addition to the birational morphism $\sigma : Y \to W$ constructed 
in Lemma {\rm \ref{canmod}}
there are a birational morphism $\tau : W \to V$, some $e \in \Q \cap (0, 1)$ 
and effective divisors $E_{\tau \sigma}$, $E_{\sigma}$, $E_{L_Y}$ on $Y$ and $E_{\tau}$ on $W$ satisfying:
\begin{enumerate}
\item[a)] $E_{\tau}$ is $\tau$-exceptional, $E_{\sigma}$ and $E_{L_Y}$ are $\sigma$-exceptional, and 
$$ E_{\tau \sigma} = E_{\sigma} + (1-e) E_{L_Y} + \sigma^*E_{\tau},
$$
\item[b)] Writing  $D_W := \sigma_*D_Y, D_V := \tau_* D_W$ etc. one has
\begin{gather*}
K_Y + D_Y + e L_Y = \sigma^*(K_W + D_W + e L_W) + E_{\sigma} + (1-e) E_{L_Y}\\
= \sigma^* \tau^*(K_V + D_V + eL_V) + E_{\tau \sigma}.
\end{gather*}
\item[c)] $\sigma^* L_W = L_Y + E_{L_Y}$, and $L_W$ is nef.
\item[d)] $e= \min\{e' \, | \, K_S + D_S + e' L_S$ is pseudo-effective$\}$.
Here $S$ can be chosen to be equal to $Y, \  W$, or  $V$ and the resulting $e$ is independent of this choice.
\item[e)] $(V, D_V)$ and hence $V$ are $\klt$.
\item[f)] One of the following holds true:
\begin{enumerate}
\item[(1)] $K_V + D_V + eL_V \equiv 0$, the Picard number $\rho(V) = 1$,
and $V$ is a $\klt$ del Pezzo (rational) surface. In particular, $-K_V$ is an ample $\Q$-divisor
and $V$ has at most quotient singularities.
\item[(2)] $V$ is the total space of a $\P^1$-fibration 
over a curve with general fibre $\Gamma$, the Picard number $\rho(V) = 2$,
and $K_V + D_V + eL_V \equiv \beta \Gamma$ for some $\beta \in \Q_{>0}$.
\end{enumerate}
\end{enumerate}
\end{proposition}

\begin{proof}
We start with the morphism $\sigma : Y \to W$ from Lemma \ref{canmod}. For $L_W := \sigma_*L_Y$ 
one has  $\sigma^*L_W = L_Y + E_{L_Y}$ where $E_{L_Y}$ is supported in the exceptional locus of $\sigma$. 
Since $L_Y$ is nef, $L_W$ is also nef, and $E_{L_Y}$ is effective. 
By Lemma \ref{canmod} one finds for all $e'$
\begin{equation} \label{eq.1}
K_Y + D_Y + e' L_Y = \sigma^*(K_W + D_W + e' L_W) + E_{\sigma} + (1 - e') E_{L_Y}.
\end{equation}
So the assertion c) and the first equation in the assertion b) hold true.

Starting from $W_0 = W$ we will construct for some $r\geq 0$ and for $i=0, \ldots, r-1$ 
a chain of birational morphisms $\tau_i : W_i \to W_{i+1}$, such that
$W_r$ satisfies the conditions stated in Proposition \ref{threshold}, f) (1) or (2). 
We will show inductively that the following conditions (c1) - (c5) hold for $i=1,\ldots,r$ and that
(c6)-(c8) hold for $i=1,\ldots,r-1$.
\begin{itemize}
\item [(c1)] $(W_i, D_i)$ is $\klt$.
\item[(c2)] $K_i + D_i$ is not pseudo-effective.
\item[(c3)] $K_i + D_i + L_i$ is ample.
\item[(c4)] $e_i = \min \{e'\in (0,1) \, | \, K_i + D_i + e' L_i$ is nef $\}$ exists and is rational.
\item[(c5)] $1 > e_0 \ge e_1 \ge \cdots \ge e_r > 0$.
\item[(c6)] $\rho(W_{i+1}) = \rho(W_i) - 1$.
\item[(c7)] $L_i$ is nef, and $\tau_i^* L_{i+1} = L_i + E_{L_i}$ for an effective $\tau_i$-exceptional divisor
$E_{L_i}$.
\item[(c8)] $K_i + D_i + e_i L_i = \tau_i^*(K_{i+1} + D_{i+1} + e_i L_{i+1})$.
\end{itemize}
Here $K_i = K_{W_i}$, and $D_i$ or $L_i$ denotes the pushdowns of $D_Y$ or $L_Y$
to $W_i$. We write $\rho(W_i)$ for the Picard number of $W_i$.

\begin{claim} \label{claim1}
(c3) and (c7) are true for all $i \ge 0$, and
(c1) and (c2) hold for $i = 0$. 
\end{claim}
\begin{proof}
Note that $\tau_i$ is birational.
(c3) and (c7) are true for $i = 0$ and hence they are true for all $i \ge 0$ on surfaces;
see Lemma \ref{canmod} and the proof for the assertion c) above.
(c1) is also part of Lemma \ref{canmod}. 
For (c2) set $e' = 0$ in the equation \eqref{eq.1} and use the non-pseudo-effectiveness of $K_Y + D_Y$.
\end{proof}

\begin{claim} \label{claim2} \
\begin{enumerate}
\item[(i)] The conditions (c2) and (c3) for some $i$ imply (c4) with $e_i\in (0,1)$.
\item[(ii)] In particular, (c4) and (c5) hold for $i=0$.
\end{enumerate}
\end{claim}
\begin{proof}
Knowing (c2) for some $i$ the condition (c3) allows to deduce from \cite[Th 4-1-1]{KMM} 
or \cite[Th 3.5]{KM} that there exists a rational number 
$$
d_i = \max\{d \, | \, (K_i + D_i + L_i) + d (K_i + D_i) \mbox{ is nef }\}.
$$
Since $K_i + D_i + L_i$ is ample, $d_i >0$. Then $e_i = 1/(1 + d_i)$.
\end{proof}

Assume now we have found the birational morphisms $\tau_i$ for $i < i_0$, that
(c1)-(c5) hold for $i=0,\ldots, i_0$ and that (c6)-(c8) hold for $i=0,\ldots, i_0-1$.

By \cite[Complement 3.6]{KM}, the condition (c2) implies the existence of 
a $K_{i_0} + D_{i_0}$-negative extremal ray $R_{i_0}$, perpendicular to $K_{i_0} + D_{i_0} + e_{i_0} L_{i_0}$.
We choose $\tau_{i_0} : W_{i_0} \to W_{i_0+1}$ to be the contraction of $R_{i_0}$
(i.e., of all the curves proportional to $R_{i_0}$). In particular, one finds
\begin{equation}\label{eq.2}
\tau_{i_0}^* \tau_{i_0\, *}(K_{i_0} + D_{i_0} + e_{i_0} L_{i_0})=K_{i_0} + D_{i_0} + e_{i_0} L_{i_0}.
\end{equation}
Suppose that $\tau_{i_0}$ is birational. Then for $i = {i_0}$ the condition
(c6) holds. 
(c8) follows from the equation \eqref{eq.2}. 

Knowing (c1)-(c8) for $i=i_0$ it is easy to verify (c1)-(c5) for $i=i_0+1$. We remark that
(c7) and (c8) for $i_0$ imply that 
$$ K_{i_0} + D_{i_0}= \tau_{i_0}^*(K_{i_0+1} + D_{i_0+1}) + e_{i_0}E_{L_{i_0}},$$
so (c1) and (c2) for $i_0 + 1$ follow from the corresponding statements for $i_0$, and
hence (c4) for $i_0+1$ follows from Claim \ref{claim2}. 

By the choice of $e_{i_0}$ 
$$
K_{i_0} + D_{i_0} + e_{i_0} L_{i_0} = \tau_{i_0}^*(K_{i_0+1} + D_{i_0+1} + e_{i_0} L_{i_0+1})
$$ 
is nef. This is possible only if $K_{i_0+1} + D_{i_0+1} + e_{i_0} L_{i_0+1}$ is nef, and hence only if
$e_{i_0} \geq e_{i_0+1}$, as claimed in (c5).

If $\tau_{i_0}$ is birational, we can continue this process. This way,
one obtains birational morphisms $\tau_j : W_j \to W_{j+1}$ ($0 \le j \le r$)
satisfying the conditions (c1) - (c8).
The condition (c6) implies that
$r < \rho(W)$.
\vspace{.2cm}
 
If $\tau_{i_0}$ is non-birational we set $V=W_{i_0}$ and $e=e_{i_0}$ 
in Proposition \ref{threshold}. The assertions a) and the second half of b) 
follow from (c5), (c7) and (c8), whereas e) is the same as (c1). It remains to verify d) and f).\vspace{.2cm}

\noindent{\bf Case (1).} If the image of $\tau_{i_0}$ is a point, we claim that
in Proposition \ref{threshold} f) we are in the first case there. By the construction, $\rho(V) = 1$.

Recall that the singularities of a $\klt$ surface are just quotient singularities.
Since $L_Y$ and hence $L_V = \tau_* \sigma_* L_Y$ can not be numerically trivial, it must be 
a positive multiple of the generator of the Neron-Severi group of $V$. So the definition of $e$
implies that $K_V + D_V + eL_V \equiv 0$. By \cite[Lemma 1.3]{GZ} a $\klt$ surface with $-K$ ample
is rational.\vspace{.2cm}

\noindent{\bf Case (2).}
We claim that the second case in Proposition \ref{threshold} f) occurs if $\tau_{i_0}$ has a curve $W_{i_0+1}$ as its image. 
Let $\Gamma$ denote a general fibre of $\tau_{i_0}$.

For $V=W_{i_0}$ one finds $\rho(V) = 1 + \rho(W_{i_0+1}) = 2$.
Our $\Gamma$ generates the extremal ray $R_{i_0}$ giving rise to the contraction $\tau_{i_0}$.
So every fibre of $V \to W_{i_0+1}$ is irreducible (also because $\rho(V) = 2$).
Since the nef divisor $K_V + D_V + eL_V$ is perpendicular to
$R_{i_0}$ and hence to the nef divisor $\Gamma$, one finds that 
$K_V + D_V + eL_V \equiv \beta \Gamma$ for some $\beta > 0$. 

Since $K_V + D_V + L_V \equiv (1-e)L_V + \beta \Gamma$ is ample, we have $\Gamma . L_V > 0$. Now
$0 = \Gamma . \beta \Gamma = \Gamma . (K_V + D_V + eL_V) > \Gamma . K_V$ and 
hence $\Gamma \cong \P^1$.\vspace{.2cm}

We still have to characterize $e$ as the pseudo-effective threshold as claimed in the assertion d) of 
Proposition \ref{threshold}.

Clearly, when $S = V$, our $K_S + D_S + eL_S \equiv \beta \Gamma$ (setting $\beta = 0$  and $\Gamma$ to be
any ample divisor, in Case (1)) is pseudo-effective, so by the assertion b) of 
Proposition \ref{threshold} the same is true when $S = Y$ or $S = W$.

Conversely, suppose that $K_S + D_S +e'L_S$ is pseudo-effective for some $e'$
and some $S \in \{Y, W, V\}$. Then the same holds for $S = V$ by considering the pushdown.

For $S = V$ we can write this divisor as $\beta \Gamma + (e'-e) L_V$.
Thus $0 \le \Gamma . (\beta \Gamma + (e'-e) L_V) = (e'-e) \Gamma . L_V$.
Since $K_V + D_V + L_V \equiv \beta \Gamma + (1-e) L_V$ is ample,
we have $\Gamma . L_V > 0$ in both Cases (1) and (2), and hence $e' \ge e$.
\end{proof}

The next two Lemmata give a universal upper bound for the threshold $e$ in 
Proposition \ref{threshold}. 

\begin{lemma}\label{e1}
In the situation considered in Proposition \rm{\ref{threshold} f)}, Case $(1)$,
there is a constant $e(b, N) < 1$, depending only on $b$ and $N$,
such that the threshold $e \le e(b, N)$.
\end{lemma}
\begin{proof}
Let $\pi : \widetilde{V} \to V$ be a minimal resolution.
So one has a commutative diagram
$$
\begin{CD}
Y @> \sigma >> W\\
@V \xi VV @V \tau VV\\
\widetilde{V} @> \pi >> V   .
\end{CD}
$$
As usual, when there is a birational morphism $Y \to S$ we will write
$D_S$ and $L_S$ for the direct images of $D_Y$ and $L_Y$, respectively.

Write $\pi^*K_V = K_{\widetilde V}+ J$ with $J$ an effective and $\pi$-exceptional $\Q$-divisor.
Note that $(V, D_V)$ and hence $V$ and $(\widetilde{V}, J)$
are $\klt$. Since $H_{Nb} := NbL_{\widetilde V}$ is 
a nef line bundle with $H_{Nb} - (K_{\widetilde V} + J)$
nef and big, \cite[Th 3.1]{Ka} tells us that $|2 H_{Nb}|$
is base point free. 

So $L_{\widetilde V}$ is $\Q$-linearly equivalent to $L_{\widetilde V}' := 
H_{2Nb}/2Nb$
for a smooth divisor $H_{2Nb} \in |2H_{Nb}|$
intersecting $D_{\widetilde V}$ transversely and away from the fundamental point
of the inverse of the birational morphism $\xi : Y \to \widetilde{V}$.

For $L_Y' := \xi^* L_{\widetilde V}'$, the pair $(Y, D_Y + L_Y')$ is $\klt$.
Write $\xi^*L_{\widetilde V} = L_Y + E$ with $E \ge 0$ $\xi$-exceptional.
Then $K_Y + D_Y + L_Y' \sim_{\Q} (K_Y + D_Y + L_Y) + E$ is big. 

So we are allowed to apply the first part of Lemma \ref{big} and we find a constant
$M(b, N)$ such that
$|(t_0+1)(K_Y + D_Y + L_Y')|$ gives rise to a birational map for all $t_0 \in \Z_{> 0}$
with $2Nb | (t_0+1)$ and $t_0 \ge M(b, N)$. 
Thus the same holds for $|(t_0+1)(K_V + D_V + L_V')|$ with
($L_V \sim_{\Q}$) $L_V'$ the pushdown of $L_Y'$.

\par
%
So $(t_0+1)(K_V + D_V + L_V) . \bar{\Gamma} \ge 1$ for any movable curve $\bar{\Gamma}$ on $V$. 
On $\widetilde{V}$, we take $\Gamma \cong \P^1$ 
with $\Gamma^2 = 0$ or $1$ (when $\widetilde{V}$ is ruled or $\P^2$)
such that $\bar{\Gamma} = \pi(\Gamma)$.
Note that $\bar{\Gamma} . K_V = \Gamma . (K_{\widetilde V} + J) \ge \Gamma . K_{\widetilde V} \ge -3$.

\par
If $e \le 1/2$ there is nothing to show. Otherwise 
$$
0 = \bar{\Gamma} . (K_V + D_V + e L_V) \ge -3 + e \bar{\Gamma} . L_V \ge -3 + \frac{1}{2} \bar{\Gamma} . L_V,$$
Then 
$$6(1-e) \ge (1-e) \bar{\Gamma} . L_V= \bar{\Gamma} . (K_V + D_V + L_V) \ge \frac{1}{t_0+1}$$
gives an upper bound for $e$.
\end{proof}

\begin{lemma}\label{e2}
In Case $(2)$ of Proposition {\rm \ref{threshold} f)}, there is a constant
$\nu = \nu(N, b)$ (depending only on $N, b$) such that
the threshold $e$ satisfies 
$$e \le 1 - \frac{1}{4 \nu} < 1.$$
\end{lemma}

\begin{proof}
Again it is sufficient to consider the case $e \ge 1/2$.
We calculate 
$$0 = \Gamma . \beta \Gamma = \Gamma . (K_V + D_V + e L_V) \ge -2 + \frac{1}{2} \Gamma . L_V = 
-2 + \frac{1}{2} \widetilde{\Gamma} . L_Y.$$
Here the fibre $\widetilde{\Gamma}$ is the pullback on $Y$ of the general fibre $\Gamma$ on $V$ in 
Proposition \ref{threshold} f), Case (2).
Since $K_Y + D_Y + L_Y$ is big and $N \widetilde{\Gamma} . (K_Y + L_Y) \in \Z_{> 0}$,
we apply \cite[Prop 6.3]{FM}, obtain $\nu = \nu(N, b)$ satisfying
the following and hence conclude the lemma
%
%
(noting that $E_{\tau \sigma}$ is contained in fibres):
$$\nu \le \widetilde{\Gamma} . (K_Y + D_Y + L_Y) = (1-e) \widetilde{\Gamma} . L_Y \le 4(1-e).$$
\end{proof}

\section{The proof of Theorem \ref{ThA} and Corollary \ref{CoB} }\label{prf}

When $K_Y + D_Y$ is big, especially when $L_Y \equiv 0$, the statement of Theorem \ref{ThA} 
has been verified in Lemma \ref{big}.  
If $L_Y$ is big the theorem follows from Lemmata \ref{smalle}, \ref{e1} and \ref{e2}. 
So for Theorem \ref{ThA} it remains to consider the case:

\begin{assumption}\label{ass}
$L_Y$ is not numerically trivial, 
$\kappa(L_Y) \le 1$, and $\kappa(K_Y + D_Y) \le 1$. 
\end{assumption}

By \cite{8aut}, for a nef $\Q$-divisor $L_Y$ on a projective manifold $Y$ there exists a {\it nef reduction},
i.e. an almost holomorphic dominant rational map $\varrho:Y \cdots \to T$ such that the restriction of $L_Y$
to compact fibres is numerically trivial, and the restriction to general curves $C$ with $\dim(\varrho(C)) >0$ is of positive degree. The dimension of $T$ is called the {\it nef dimension} and denoted as $n(L_Y)$. Obviously
$\kappa(L_Y) \leq n(L_Y)$ and the first assumption in \ref{ass} implies that the nef dimension $n(L_Y)$ is one or two. In the first case the reduction map $\varrho$ is birational, whereas in the second case it is
a morphism to a curve.

Recall that starting with \ref{wlg}, we had chosen $Y$ such that the birational morphism
$\sigma : Y \to W$, constructed in Lemma \ref{canmod}, is a minimal desingularization.
As in Section \ref{log} $L_W$ and $D_W$ are the direct images of $L_Y$ and $D_Y$, respectively.
We write $K_Y=\sigma^*K_W - J$ with $J$ an effective $\sigma$-exceptional $\Q$-divisor.

\begin{lemma}\label{easy} \ 
\begin{itemize}
\item[(1)] $0 \le L_Y^2 \le L_W^2.$
\item[(2)] If $n(L_Y)=2$, then $\kappa(L_Y)\leq 0$ and $L_Y.K_Y \ge 0$. 
\item[(3)] Let $e$ be the threshold from Proposition \rm{\ref{threshold}} and 
let $P_Y$ be the positive and $N_Y$ the negative part in the Zariski decomposition
$$
K_Y + D_Y + L_Y = P_Y + N_Y.
$$
Then $P_Y - (1-e) \sigma^*L_W$ is pseudo-effective.
Furthermore, 
\begin{equation}\label{eqP}
P_Y^2 \ge (1-e)^2 L_W^2 \ge (1-e)^2 L_Y^2.
\end{equation}
\end{itemize}
\end{lemma}

\begin{proof}
(1) Recall that $L_Y$ is nef, hence $L_W$ as well. Since $L_Y \leq \sigma^* L_W$ one obtains (1).

(2) We may assume that $K_Y$ is not pseudo-effective and 
hence $Y$ is a (birational) ruled surface.
Since $L_Y$ is nef, but neither big nor numerically trivial, one finds that $L_Y^2=0$ and that $H^2(Y,m L_Y)=0$ for $m$ sufficiently large. Then for $m$ sufficiently divisible
$$
\dim(H^0(Y,mL_Y)) \geq -\frac{L_Y.K_Y}{2}\cdot m + \chi(\OO_Y).
$$
If $L_Y.K_Y < 0$, then $\kappa(L_Y)=1$ and the restriction of $L_Y$ to the fibres of $\Phi_{|mL_Y|}$ is numerically trivial, contradicting $n(L_Y)=2$. 

(3) Using the notations from Lemma \ref{canmod}, $P_Y = \sigma^*(K_W + D_W + L_W)$ and $N_Y = E_{\sigma}$.
Moreover $K_Y + D_Y + eL_Y$ is pseudo-effective by the choice of $e$. Then
its $\sigma$-pushdown $K_W + D_W + eL_W$ is pseudo-effective as well and one obtains the first part of (3). Since $P_Y$ and $L_W$ are nef, the pseudo-effectivity of  $P_Y - (1-e) \sigma^*L_W$ implies
\eqref{eqP}. 
\end{proof}

\begin{lemma}\label{LK}
Assume that $L_Y$ is not numerically trivial, that $\kappa(L_Y) \le 1$ and that 
either $L_Y^2 > 0$ or $L_Y . K_Y \ge 0$. Then $|MK_X|$ is an Iitaka fibration
for some $M = M(b, N, e)$ depending only on $b, N, e$.
\end{lemma}

\begin{proof}
Keeping the notations from Lemma \ref{easy} (3), one has:
\begin{claim}\label{3.4}
 \ \hspace*{\fill} $\displaystyle P_Y^2 \, \ge \, \frac{(1-e)^2}{3 (N b)^2}.$ \hspace*{\fill} \
\end{claim}
\begin{proof} 
Assume first that $L_Y^2 > 0$. Since $Nb L_Y$ is 
%
%
an integral Cartier divisor,
$L_Y^2 \ge 1/(Nb)^2$ and the claim follows from \eqref{eqP} in Lemma \ref{easy} (3).

Assume next that $L_Y^2 = 0$. Since $L_Y$ is nef, one finds by assumption that
$$
L_W . K_W = L_Y . (K_Y + J) \ge L_Y . K_Y \ge 0.
$$
If $L_Y . K_Y$ is positive, by Lemma \ref{form} (5) 
it has to be larger than or equal to $1/Nb$. Applying Lemma \ref{easy} (3) one finds
\begin{multline*} 
P_Y^2 \ge P_Y . (1-e) \sigma^* L_W
= (1-e) (K_W + D_W + L_W) . L_W \\
\ge (1-e) L_W . K_W \ge
(1-e) L_Y . K_Y \ge \frac{1-e}{Nb}\ge \frac{(1-e)^2}{3 (N b)^2}.
\end{multline*}
If $L_Y . K_Y = 0$, consider first the case
$L_Y . \sigma' D_W > 0$, where $\sigma'$ stands for the proper transform. 
By Lemma \ref{form} this intersection number is $\ge 1/(Nb)^2$. As above one obtains
$$P_Y^2 \ge (1-e) L_W . D_W = (1-e) L_Y . \sigma^*D_W
\ge (1-e) L_Y . \sigma' D_W \ge \frac{1-e}{(Nb)^2}\ge \frac{(1-e)^2}{3 (N b)^2}.$$
It remains to handle the worse case 
\begin{equation}\label{worse}
L_Y^2=L_Y . K_Y=L_Y . \sigma' D_W = 0.
\end{equation}
Since $K_Y + D_Y + L_Y$ is big and since $L_Y$ is not numerically trivial,
$$
0 < L_Y . (K_Y + D_Y + L_Y)
= L_Y . D_Y.
$$
Thus $L_Y . D_1 > 0$ for some irreducible curve $D_1$ in $\Supp D_Y \setminus \sigma' D_W$.
Then $D_1$ lies in the exceptional locus of $\sigma$, hence $D_1\cong \P^1$ and $D_1^2 \le -2$.
  
If $D_1^2 = -n$ with $n \ge 3$, then \cite[Lemma 1.7]{Z} implies that
\begin{gather*}
J \ge \frac{n-2}{n} D_1 \ge \frac{1}{3} D_1,\mbox{ \ \ and}\\
P_Y^2 \ge (1-e) L_W . K_W = (1-e) L_Y . (K_Y + J)
\ge (1-e) L_Y . \frac{1}{3} D_1 \ge \frac{1-e}{3Nb}\ge \frac{(1-e)^2}{3 (N b)^2}.
\end{gather*}

If $D_1^2 = -2$ consider the contraction $\sigma_1 : Y \to W_1$ of $D_1$.
Then $\sigma_1^* L_{W_1} = L_Y + a D_1$
with $a = L_Y . D_1/2 \ge 1/2Nb$.
Note that $0 = L_Y^2 = (\sigma_1^*L_{W_1} -aD_1)^2 = L_{W_1}^2 - 2a^2$,
so $L_W^2 \ge L_{W_1}^2 = 2a^2 \ge 1/2(Nb)^2$ and
$$
P_Y^2 \ge (1-e) L_W^2 \ge \frac{1-e}{2(Nb)^2} \ge \frac{(1-e)^2}{3 (N b)^2}.
$$ 
\end{proof}

As a next step in the proof of Lemma \ref{LK} consider two general points $x_1, x_2$ of $Y$.
If the nef dimension $n(L_Y) = 1$, we may assume that the two points are not in the same fibre of the nef reduction. Thus for a very general curve $C_t$ on $Y$ containing $x_1, x_2$, one has
$$
P_Y . C_t \ge (1-e) L_Y . C_t \ge \frac{1-e}{Nb}.
$$
Then the adjoint linear system
$$|K_Y + \lceil s_0(K_Y + D_Y + L_Y) + L_Y \rceil|, \hskip 1pc
\text{for} \hskip 1pc 
s_0 = b(\lceil \frac{5Nb}{1-e}\rceil +1)-1$$
separates the points $x_1, x_2$. In fact,
the nef part of the divisor 
$$\lceil s_0(K_Y + D_Y + L_Y) + L_Y \rceil$$
is larger than $s_0 P_Y$ and the inequalities
$$
s_0 P_Y . C_t \ge s_0(1-e) L_Y . C_t \ge \frac{s_0(1-e)}{Nb} \ge 4 \mbox{ \ \ and \ \ }
(s_0 P_Y)^2 \ge s_0^2 \frac{(1-e)^2}{3(Nb)^2} > 8
$$
allow to apply \cite[Th 3.2]{La}.
Thus, by Lemma \ref{form},
$$h^0(X, (s_0+1)K_X) = h^0(Y, (s_0+1)(K_Y + D_Y + L_Y)) \ge $$
$$h^0(X, K_Y + \lceil s_0(K_Y + D_Y + L_Y) + L_Y \rceil) \ge 2.$$
Now by \cite[Th 4.6]{Ko86}, 
$\Phi_{|tK_X|}$ is an Iitaka fibration for
$t = (s_0+1)(2M + 1) + M$, where $M$ is a constant as in
\cite[Corollary 6.2]{FM}, depending only on $A(b, N)$.
\end{proof}

Recall that Assumption \ref{ass} implies that $n(L_Y)$ is one or two. In the second case,
Lemma \ref{easy} (2) allows to apply Lemma \ref{LK}. So it remains to consider
the case below:

\begin{lemma}\label{n=1}
Assume that $n(L_Y) = 1$, $L_Y^2 = 0$ and $L_Y . K_Y < 0$.
Then $Y$ is a ruled surface over a curve $C$ of genus $q(Y)$ with general fibre 
$\Sigma \cong \P^1$. The $\Q$-divisor $L_Y$ is $\Q$-linearly equivalent to a positive multiple of $\Sigma$, and
$|MK_X|$ is an Iitaka fibration for some constant $M = M(b, N, q(Y))$ depending only on
$b, N, q(Y)$.
\end{lemma}

\begin{proof} Since $n(L_Y)=1$ the nef reduction is a map $\varrho: Y\to C$ to a curve.
It is an easy exercise (whose solution can be found in \cite[Proposition 2.11]{8aut})
to see that $\varrho$ is a morphism and that $L_Y$ is numerically equivalent to a positive multiple of the general fibre $\Sigma$ of $\varrho$. Since $L_Y . K_Y < 0$, one has $\Sigma . K_Y < 0$. Thus
$2g(\Sigma) - 2 = \Sigma . K_Y < 0$ and $\Sigma \cong \P^1$.
So $\varrho:Y \to C$ is a $\P^1$-fibration and $Y$ is a ruled surface
with $g(C) = q(Y)$. 

Moreover the divisor $NbL_Y$ on the ruled surface $Y \to C$ is numerically equivalent to 
some $\alpha \Sigma$. Considering the intersection of $\Sigma$ with a section of $\varrho:Y \to C$
one sees that $\alpha$ is an integer. The numerically trivial sheaf $NbL_Y-\alpha \Sigma$  
is linearly equivalent to
the pullback of a numerically trivial sheaf on a relative minimal model of $Y$ which in turn must be the pullback of a sheaf on $C$. Hence $NbL_Y \sim \varrho^*\Pi$ 
for some integral divisor $\Pi$ on $C$
of positive degree. Then $(2g(C)+1)\Pi$ is very ample and 
$(2g(C)+1)NbL_Y$ is linearly equivalent to the disjoint union $H$ of smooth fibres
in general position. For ($L_Y \sim_{\Q}$) $L'_Y=H/(2g(C)+1)Nb$ the pair $(Y, D_Y + L_Y')$ is $\klt$ and 
$K_Y + D_Y + L_Y'$ is big. The coefficients of $D_Y + L_Y'$ lie in the $\DCC$ set
$A(b, N) \cup \{1/(2g(C)+1)Nb\}$. Hence repeating the arguments used in the proof of Lemma \ref{big}
one finds a constant $M$ depending only on $b,N$ and $g(C)$ such that $|(s+1)(K_Y + D_Y + L_Y)|$
defines a birational map for all $s\geq M$ with $s+1$ divisible by $(2g(C)+1)Nb$.  
Now the lemma follows from Lemma \ref{form} (4).
\end{proof}

\begin{lemma}\label{q} Keeping the assumptions made in Lemma \rm{\ref{n=1}},
either $K_Y + D_Y$ is big or $q(Y) \leq 1$.
\end{lemma}

\begin{proof}
If $K_Y + D_Y$ is not pseudo-effective we can apply Proposition \ref{threshold} f).
There, in Case (1) the irregularity is zero. So we only have to consider Case (2).
Using the notations introduced there, 
$K_Y + D_Y + L_Y \equiv (1-e)L_Y + \beta \Gamma + E_{\tau \sigma}$
is big,
$\Gamma . L_Y > 0$ and hence, using the notation from Lemma \ref{n=1}, $\Gamma . \Sigma > 0$.
Further $\Gamma \cong \P^1$.
%
So there are two different $\P^1$-fibrations on $Y$ with fibres $\Gamma$ and $\Sigma$, 
and $Y$ is rational.

Therefore, we may assume that $K_Y + D_Y$ is pseudo-effective with $\kappa(K_Y + D_Y) \le 1$. 
Applying Construction \ref{LMMP} to the $\klt$-pair $(Y, D_Y)$,
we get morphisms $\gamma:Y\to Z$ and $\psi_Z:Z\to B$ with $\gamma$ birational, and an ample 
$\Q$-divisor $H$ on $B$ such that
$$
E_{\gamma}=K_Y + D_Y - \gamma^*\psi_Z^*(H)
$$
is an effective $\gamma$-exceptional $\Q$-divisor consisting of rational curves. By the assumption
$$ \dim B = \kappa(H) = \kappa(K_Y + D_Y) \leq 1.$$

\par
Consider the case $\dim B = 1$. So $\psi = \psi_Z \circ \gamma : Y \to B$ is a family of curves 
over a curve with general fibre $\Gamma$. By abuse of notation $\Gamma$ will also be 
considered as the general fibre of $\psi_Z$. For $\alpha  = \deg H$, one has
$$K_Y + D_Y \sim_{\Q} \, \equiv \alpha \Gamma + E_{\gamma} \mbox{ \ \ and \ \ } 
K_Z + D_Z \equiv \alpha \Gamma .$$
Since $E_{\gamma}$ is contained in fibres 
$$0 = \Gamma . (\alpha \Gamma + E_{\gamma}) = \Gamma . (K_Y + D_Y) \ge \Gamma . K_Y$$
and $\Gamma$ is either $\P^1$ or an elliptic curve.

Since $K_Y + D_Y + L_Y \equiv \alpha \Gamma + E_{\gamma} + L_Y$ is big,
$\Gamma . L_Y > 0$. Using the notations from Lemma \ref{n=1}, this implies that $\Gamma . \Sigma > 0$ where $Y$ is ruled over $C$ with general fibre $\Sigma$. So $\Gamma$ dominates the base curve $C$
and $q(Y)=g(C)\leq 1$, as claimed.

In case $\dim B = 0$ one has $K_Z + D_Z \equiv 0$ (indeed, $\sim_{\Q} 0$ by \cite{Ko+}).
If $L_Y . E_{\gamma} = 0$, then $K_Y + D_Y + L_Y \equiv L_Y + E_{\gamma}$ is the
Zariski decomposition and hence 
$$2 = \kappa(K_Y + D_Y + L_Y) = \kappa(L_Y) \le n(Y) = 1,$$
a contradiction.

Thus $L_Y . E_{\gamma} > 0$ and, using the notations from Lemma \ref{n=1},
one finds $\Sigma . E_{\gamma} > 0$. The divisor $E_{\gamma}$ is exceptional for
the birational morphism to the $\klt$ surface $Z$, whence all its components are isomorphic to $\P^1$. 
Since one of them intersects $\Sigma$, the base curve $C$ in Lemma \ref{n=1} is dominated by $\P^1$
and hence $g(C)=q(Y) = 0$.
\end{proof}

\begin{proof}[Proof of Theorem \rm{\ref{ThA}}]
As recalled at the  beginning of this section it remains to verify the theorem
under Assumption \ref{ass}. Then the theorem follows from
Lemmata \ref{LK}, \ref{n=1}, and \ref{q}, using Lemmata \ref{e1} and \ref{e2}.
\end{proof}

\begin{proof}[Proof of Corollary \rm{\ref{CoB}}]
When $\kappa(X) = 0$, one can take $M_3$ to be the Beauville number
as in \cite[\S 10]{Mo}. When $\kappa(X) = 1$, the result is just \cite[Corollary 6.2]{FM}.
When $\kappa(X) = 3$, we can take $M_3 = 77$ by \cite[Th 1.1]{CC} (see also \cite{HM}, \cite{Ta}).
So the only remaining case is the one where $\kappa(X) = 2$.
Here the corollary follows from Theorem \ref{ThA}
for $n = 3$, $b = 1$, $B_{n-2} = 2$ and $N = N(B_{n-2})= 12$.
\end{proof}

\section{Some comments}\label{com}

\begin{remark}\label{end}
Although the arguments used in Sections \ref{log} and \ref{prf} are formulated 
just for surfaces $Y$ some can be easily extended to the higher dimensional case.
In particular, the general minimal model program in \cite{BCHM} extends the Zariski decomposition
for pseudo-effective divisors to the case $\dim(Y)>2$.

However, as pointed out by the referee, there is no replacement for the fact that on surfaces the direct image of a nef divisor is nef. Similarly, the proof of Claim \ref{3.4}, essential for 
Lemma \ref{LK} is done ``case by case''. In the ``worse case'' \eqref{worse} 
one can not extract any positivity from $L_Y$, nor from $K_Y +\sigma' D_W$. One has to get the positivity
from exceptional components of $D_Y$. Again, there is little hope to do something similar
for $\dim(Y)>2$.  
\end{remark}
\begin{remark}\label{rem2.1} 
If the general fibre $F$ of the Iitaka fibration is a good minimal model, and hence if
$\omega_F^b\cong \OO_F$, choose an ample invertible
sheaf $\AA$ on $X$. Writing $h$ for the Hilbert polynomial of $\AA|_F$ one obtains a morphism
from the complement $Y_0$ of the discriminant locus to the moduli scheme
$M_h$ of polarized minimal models with Hilbert polynomial $h$. In \cite{Vi06} 
we constructed a compactification $\overline{M}_h$ of $M_h$ and a nef $\Q$-Cartier divisor $\lambda$ on $\overline{M}_h$,
which on each curve meeting $M_h$ corresponds to the semistable part. Moreover, $\lambda$ is 
ample with respect to $M_h$. In different terms, there is an effective $\Q$-Cartier divisor $\Gamma$
on $\overline{M}_h$, supported in $\overline{M}_h\setminus M_h$, such that $\alpha\lambda-\Gamma$ is ample
for all $\alpha\geq 1$.

If we choose $Y$ such that $Y_0\to M_h$ extends to a morphism $\varphi:Y\to \overline{M}_h$, one finds that
$L_Y=\varphi^*(\lambda)$, and hence that $\alpha L_Y-\varphi^*(\Gamma)$ is 
semi-ample. For some constant $C>0$, depending only on $h$, both $CL_Y$ and 
$C(L_Y-\varphi^*(\Gamma))$ are divisors, and
choosing $C$ large enough $C(\alpha L_Y-\varphi^*(\Gamma))$ will have lots of global sections for
all $\alpha\in \Z_{>0}$ (provided that $L_Y$ is not numerically trivial). 
Perhaps this allows us to answer Question \ref{quest} 
in the affirmative, assuming the existence of good minimal models. 
Perhaps to this aim one has to compare the image of the map $\varphi:Y\to \overline{M}_h$
with the log minimal model of $Y$.

Anyway, one still would need some argument which guarantees the independence 
of the constant $M$ from the Hilbert polynomial $h$.
\end{remark}

\begin{remark}\label{rem2.2}
Gianluca Pacienza \cite{Pa} recently gave an affirmative answer to Question \ref{quest}
for $\kappa=n-2$, or more precisely if the general fibre $F$ of the Iitaka fibration has a good 
minimal model, assuming that $Y$ is non-uniruled and that the morphism
$Y_0\to M_h$ in Remark \ref{rem2.1} is generically finite over its image.
Note that the last assumption implies that $L_Y$ is big.
\end{remark}

Finally let us give a direct proof of the Corollary \ref{CoB}, without referring to Theorem \ref{ThA},
but using the existence of good minimal models in dimension three:

\begin{proof}[Proof of Corollary \rm{\ref{CoB}}, using the existence of minimal models] \ \\
As before all cases are known, except the one where $\kappa(X)=2$.
Assume that $X$ is a good minimal threefold. 
For some $m \gg 0$ the morphism $\sigma: X \to S$ associated with $|mK_X|$ has connected fibres and,
by the Abundance theorem for threefolds, $K_X \sim_{\Q} \, \sigma^* G$ for an ample $\Q$-Cartier divisor $G$.

As in \cite[Proof of Corollary (0.4)]{Na} $(S, \Delta)$ is $\klt$ for the effective $\Q$-divisor
$$
\Delta := \frac{1}{12} H' + \sum_j a_j D_j' + \sum_i (1 - \frac{1}{m_i}) \Gamma_i' \mbox{ \ \ and \ \ }
K_X \sim_{\Q} \, \sigma^*(K_S + \Delta).
$$
Here $H', D_j', \Gamma_i'$ stand for the divisors $\mu_*H$, $\mu_*D_j$, $\mu_*\Gamma_i$ in the notation of \cite{Na}. Moreover
$$a_j \, \in \, K_2 := \, \{
\frac{1}{2}, \,
\frac{1}{6}, \,
\frac{5}{6}, \,
\frac{1}{4}, \,
\frac{3}{4}, \,
\frac{1}{3}, \,
\frac{2}{3}
\}.$$
We remark that by \cite[Th 3.5.2]{KM}
$S$ is the canonical model, denoted by $W$ in Lemma \ref{canmod}, and that $\Delta= L_W + D_W$.
Note that $\Delta$ has coefficients in the $\DCC$ set
$$\Ell := \{1 - \frac{1}{m} \, | \, m \in \Z_{\ge 2}\} \, \cup \, K_2 \, \cup \, \{ \frac{1}{12}\}.
$$
By \cite[Th 5.4]{La} there exists a computable constant $M$, depending only on the $\DCC$ set
$\Ell$, such that the adjoint linear system $|K_S + \lceil t(K_S + \Delta) \rceil|$
gives rise to a birational map for all $t \ge M$.
This adjoint divisor is smaller than or equal to $(t+1) (K_S + \Delta)$ provided that $12 | t$.
So $\Phi_{|(t+1)(K_S + \Delta)|}$ is birational and hence $\Phi_{|(t+1)K_X|}$ an Iitaka fibration
(see Lemma \ref{form} (4)). 
\end{proof}

\end{document}